\documentclass{amsart}
\usepackage{amsfonts}
\usepackage{amsmath,amscd}
\usepackage{amssymb}
\usepackage{amsthm}
\usepackage{newlfont}

\long\def\alert#1{\parindent2em\smallskip\hbox to\hsize%
{\hskip\parindent\vrule%
\vbox{\advance\hsize-2\parindent\hrule\smallskip\parindent.4\parindent%
\narrower\noindent#1\smallskip\hrule}\vrule\hfill}\smallskip\parindent0pt}
 \newtheorem{thm}{Theorem}[section]
 \newtheorem{cor}[thm]{Corollary}
 \newtheorem{lem}[thm]{Lemma}
 \newtheorem{prop}[thm]{Proposition}
 \theoremstyle{definition}

 \newtheorem{rem}[thm]{Remark}

 \numberwithin{equation}{section}

\begin{document}

\title[Decomposition of  the nonabelian tensor product of   ... ]
 {Decomposition of  the nonabelian tensor product of  Lie algebras via the diagonal ideal}

\author[P. Niroomand]{Peyman Niroomand}
\address{School of Mathematics and Computer Science\\
Damghan University, Damghan, Iran}
\email{niroomand@du.ac.ir, p$\_$niroomand@yahoo.com}
\author[F. Johari]{Farangis Johari}
\address{Department of Pure Mathematics\\
Ferdowsi University of Mashhad, Mashhad, Iran}
\author[M. Parvizi]{Mohsen Parvizi}
\address{Department of Pure Mathematics\\
Ferdowsi University of Mashhad, Mashhad, Iran}
\email{parvizi@math.um.ac.ir}
\author[F.G.  Russo]{Francesco G. Russo}
\address{Department of Mathematics and Applied Mathematics\endgraf
University of Cape Town\endgraf
Private Bag X1, 7701, Rondebosch\endgraf
Cape Town, South Africa}
\email{francescog.russo@yahoo.com}


\date{\today}


\begin{abstract} We prove a theorem of splitting for the nonabelian tensor product $L \otimes N$ of a pair $(L,N)$ of Lie algebras $L$ and  $N$ in terms of its diagonal ideal $L \square N$  and of the nonabelian exterior product  $L \wedge N$. A similar circumstance was described two years ago by the second author in the special case $N=L$. The interest is due to the fact that the size of $L \square N$ influences strongly the structure of $L \otimes N$. Another question, often related to the structure of $L \otimes N$,  deals with the behaviour of the operator $\square$ with respect to the formation of free products. We answer with another theorem of splitting even in this case, noting  some connections with the homotopy theory.
\end{abstract}
\maketitle

\section{Introduction and motivation}

Following \cite{ellis1, ellis2, n4}, the \textit{Schur multiplier of the pair} $(L,N)$, where $L$ is a Lie algebra with ideal $N$, is the abelian Lie algebra $M(L,N)$ which appears in the following  natural exact sequence of Mayer--Vietoris type
\[H_3(L,\mathbb{Z}) \longrightarrow H_3(L/N,\mathbb{Z}) \longrightarrow M(L,N)  \longrightarrow M(L) \longrightarrow M(L/N) \longrightarrow \]
\[ \longrightarrow \frac{L}{[L,N]} \longrightarrow \frac{L}{[L,L]} \ \longrightarrow \ \frac{L}{[L,L] + N} \ \longrightarrow 0, \]
in which we may note  the third homology Lie algebras $H_3(L,\mathbb{Z})$ and $H_3(L/N,\mathbb{Z})$ with coefficients in the ring $\mathbb{Z}$ of the integers and the two homology Lie algebras $M(L)=H_2(L,\mathbb{Z})$ and $M(L/N)=H_2(L/N,\mathbb{Z})$, again over $\mathbb{Z}$. The notion of $M(L,N)$ is in fact  categorical and can be formulated both for groups and Lie algebras (see \cite{ri1, ri2, sam, ellis1, ellis2,  ellis3, ellis4, n4}).
On another hand, the \textit{nonabelian tensor product} $L \otimes N$ of $L$ and  $N$ is defined  by generators and relations via suitable actions  of $L$ on $N$, and viceversa. Following \cite{ellis1, ellis2, guin1}, an action of $L$ on $N$ ($L$ and $N$ are supposed to be over the same field $F$) is an $F$-bilinear map   $(l,n) \in L \times N \longmapsto ~^ln \in N$ satisfying $^{[l,l']}n={~^l}(^{l'}n)-{~^{l'}}(^{l}n)$ and $~^{l}[n,n']=[~^{l}n,n']+[n,~^{l}n']$ for all $l,l' \in L$ and $n, n' \in N$. Note that the Lie multiplication induces the action of $L$ on $N$, in fact, $L$ acts on $N$ via $~^ln=[l,n]$. Since $L$ and $N$ act on each other, and on themselves, by Lie multiplication, we say that their actions are  \textit{compatible}, when $~^{~^nl}n'=~^{n'}(~^ln)$ and $~^{~^ln}l'=~^{l'}(~^nl)$ for all $l,l' \in L$ and $n,n' \in N$.  Now $L \otimes N$  is the Lie algebra generated by the symbols $l\otimes n$ with defining
relations $c(l \otimes n)=cl \otimes n = l \otimes cn,$ $(l+l')
\otimes n = l \otimes n + l' \otimes n,$ $l \otimes (n+n') = l
\otimes n + l \otimes n',$ $~^ll'  \otimes n = l \otimes ~^{l'}n -
l' \otimes ~^ln,$ $ l \otimes  ~^nn'= ~^{n'}l  \otimes n - ~^nl
\otimes n',$  $[l\otimes n, l' \otimes n']=- ~^nl \otimes
~^{l'}n'$, where $c \in F$, $l,l' \in L$ and $n,n' \in N$.  In case $L=N$ and all actions are given by Lie multiplication, $L\otimes L$ is called \textit{nonabelian tensor
square} of $L$. Note that the nonabelian tensor product always exists  and, in particular, we find the usual abelian tensor product $L \otimes_\mathbb{Z} N$, when $L$ and $N$
are abelian and the actions are all by Lie multiplication, compatible and trivial.

This construction plays a fundamental role in algebraic topology and homology; the reader may refer to  \cite{n1, n2, n3, n4} for  recent contributions in the theory of Lie algebras, but also to \cite{ellis3, ellis4, guin2,   mas1, fra} for analogies with the context of groups.  In particular, \cite{ellis4, guin1, n2} inspired most of the ideas that we are going to show in the present paper. We recall from \cite{ellis1, ellis2, n1, n4, simson}   that it is possible to get the following commutative diagram:
\begin{equation}\label{diag}\begin{CD}
@.  0 @. 0\\
@. @VVV   @VVV\\
\Gamma\left(\frac{N}{[N,L]} \right) @>>> J_2(L,N) @>>>M(L,N)@>>>0\\
@| @VVV @VVV @.\\
\Gamma\left(\frac{N}{[N,L]} \right) @>\psi>> L \otimes N @>\varepsilon_{L,N}>> L \wedge N @>>>0 @. \\
@. @V\kappa_{L,N} VV   @V\kappa'_{L,N}VV\\
@. [L,N] @= [L,N]\\
@. @VVV   @VVV\\
@. 0 @. 0\\
\end{CD}\end{equation}
where $\kappa_{L,N} : l \otimes n \in L \otimes N \longmapsto \kappa_{L,N}(l \otimes n)=[l,n] \in [L,N]$ is an epimorphism of Lie algebras such that $J_2(L,N)=\ker \kappa_{L,N} \subseteq Z(L \otimes N)$. Note that   $J_2(L,L)=J_2(L)$ was described in \cite[pp.109--110]{ellis2} when $N=L$. On the other hand,  \[L \square N=
\langle n \otimes n \ | \ n \in L \cap N\rangle\] is an ideal of
$L \otimes N$ contained in $Z(L \otimes N)$, called \textit{diagonal ideal} of $L \otimes N$. Its properties are discussed in \cite{ellis1, ellis2, n1, n2, n4} when $L=N$. The natural epimorphism $\varepsilon_{L,N} : l \otimes n \in L \otimes N \mapsto (l\otimes n)+ L \square N \in L \otimes N/L \square N$ allows us to form the Lie algebra quotient \[L\wedge N=\frac{L \otimes
N}{L \square N}=\langle l \otimes n + (L \square N) \ | \ l \in
L, n \in N \rangle =\langle l \wedge n \ | \ l \in L, n \in N\rangle,\]
called \textit{nonabelian exterior product} of $L$ and $N$. The natural epimorphism
$\kappa'_{L,N} : l \wedge n \in L \wedge N \longmapsto \kappa'_{L,N}(l \wedge n)=[l,n] \in [L,N]$ is
such that  $M(L,N) \simeq \ker \kappa'_{L,N} \subseteq Z(L \wedge N)$:  this is one of many connections between the theory of Schur multipliers of pairs and that  of nonabelian tensor products. The reference give more details on this aspect. We also note that the columns of \eqref{diag} are central extensions, while the rows of \eqref{diag} form two long exact sequences. In fact $\Gamma$ denotes the \textit{quadratic Whitehead functor}  and $\psi$ the \textit{quadratic Whitehead function}, properly defined in  \cite[Definition, p.107]{ellis2} (see also   \cite{simson} for a categorical definition of $\Gamma$ and $\psi$).

Since the notion of \textit{dimension} for a Lie algebra is in a certain sense ``more geometric than algebraic'',
we cannot expect full analogies with respect to the results in \cite{ellis3, guin2, mas1, fra}, when we replace this notion with that of order of a group. In fact the second and the fourth author have investigated the role of the homological invariants between Lie algebras and finite groups, making specific studies in \cite{n1, n2, n3, n4} during the last years. We continue  the same line of research. In Section 2 we prove a theorem of splitting for $L \otimes N$ via $L \square N$ and $L \wedge N$. This generalizes results in \cite{n2} when $N=L$. Section 3 describes the role of $\square$ with respect to the operator $*$ of free product: another theorem of splitting is shown here.

\section{Splitting of $\otimes$ via $\square$ and $\wedge$}

We begin with elementary facts, which follow from the exactness of \eqref{diag}. Our first lemma generalizes \cite[Proposition 17]{ellis2} when $N \neq L$.

\begin{lem}\label{1}
Let $N$ be an ideal of a Lie algebra $L$. If $N/[N,L]$ is free, then
\[\begin{CD} 0 @>>>\Gamma\left(\frac{N}{[N,L]} \right) @>\psi>> L \otimes N @>\varepsilon_{L,N}>> L \wedge N @>>>0 \end{CD}\]
is an exact sequence.
\end{lem}
\begin{proof}
We just need to prove that
\[\psi: \gamma(n+[N,L]) \in \Gamma \left(\frac{N}{[N,L]} \right) \longmapsto  n\otimes n \in L \otimes N \]  is injective. Let
$\theta: L \otimes N \rightarrow L/[N,L] \otimes N/[N,L]$  be the homomorphism induced by the natural projection $l \in L \mapsto l+[N,L] \in  L/[N,L]$. Then \[ \frac{L}{[N,L]} \otimes  \frac{N}{[N,L]} \simeq \frac{L}{[N,L]} \otimes_{\mathbb{Z}} \frac{N}{[N,L]}\] is a free Lie algebra, and so is $\Gamma (N/[N,L])$. Moreover, the composition $\theta \circ \psi$ maps a basis of $\Gamma (N/[N,L])$
injectively into a part of a basis of $L/[L,N] \otimes N/[L, N]$. Thus $\theta \circ \psi$ is injective. This is enough to conclude that $\psi$ is injective.
\end{proof}

In general, the columns of \eqref{diag} are short exact sequences, but not the rows. However, if $N/[N,L]$ is free, then Lemma \ref{1} allows us to conclude that even the rows of \eqref{diag} become short exact sequences.

\begin{cor}\label{2}
Let $N$ be an ideal of a Lie algebra $L$ such that $N/[N,L]$ is free. Then the following
\[\begin{CD}
0 @>>> \Gamma\left(\frac{N}{[N,L]} \right) @>>> J_2(L,N) @>>>M(L,N)@>>>0\\
@. @| @VVV @VVV @.\\
0 @>>>\Gamma\left(\frac{N}{[N,L]}\right) @>>> L \otimes N @>>> L \wedge N @>>>0 @. \\
\end{CD}\]
is a commutative diagram with short exact sequences as rows.
\end{cor}

\begin{proof} Application of Lemma \ref{1} to \eqref{diag}.
\end{proof}

A crucial step, which is fundamental for our aims, deals with the description of the natural epimorphism
\begin{equation}\label{epi1}\pi: l  \otimes  n \in  L  \otimes  N \mapsto (l + [N,L])  \otimes  (n + [N,L]) \in  \frac{L}{[N,L]}  \otimes  \frac{N}{[N,L]}
\end{equation}
and its restriction
\begin{equation}\label{epi2}\pi_|: n \ \square \ n \in  L \ \square \ N \mapsto (n + [N,L]) \ \square \ (n + [N,L]) \in  \frac{L}{[N,L]} \ \square \ \frac{N}{[N,L]}.\end{equation}
The  kernel of \eqref{epi1} is studied in the next result. We inform the reader that we will use the notion of \textit{Lie pairing} in the next proof. This can be found in  \cite[pp.101--102]{ellis2}, together with its fundamental properties.

\begin{lem}\label{5}
Let $N$ be an ideal of a Lie algebra $L$, $\pi$ as in \eqref{epi1} and  \[M= (L \otimes  [N,L]) + ([N,L] \otimes N).\]  Then  $\ker \pi= M.$
\end{lem}
\begin{proof}
Since  $L \otimes N \supseteq M$  and $M$ is an ideal of $L\otimes N$,   $M$ induces the homomorphism
\[\bar{\pi}: (l\otimes n) + M \in  \frac{L \otimes N}{M}  \mapsto (l + [N,L]) \otimes (n + [N,L]) \in \frac{L}{[N,L]} \otimes  \frac{N}{[N,L]}.\]
On the other hand, define
\[ \alpha: (l+[N,L],n+[N,L]) \in \frac{L}{[N,L]} \times  \frac{N}{[N,L]} \mapsto  (l\otimes n) + M \in  \frac{L \otimes N} {M}.\]
It is easy to check that $\alpha$ is well defined. Now for all $l_1,l_2 \in L$ and $n_1,n_2\in N$
\[\alpha(  [l_1+[N,L],l_2+[N,L]] \ , \ n_1+[N,L]  )\]
\[=([l_1,l_2] \otimes n_1) + M = (l_1\otimes [l_2,n_1]-l_2\otimes [l_1,n_1]) + M\]
\[=\alpha(l_1+[N,L],[l_2,n_1]+[N,L])-\alpha(l_2+[N,L],[l_1,n_1]+[N,L]).\]
Similarly, it is easy to see that
\[\alpha(   ^{n_1+[N,L]} l_1+[N,L] \ , \ \ ^{l_1+[N,L]}n_2+[N,L]   )\]
\[=-[\alpha(l_1+[N,L],n_1+[N,L]) \ , \ \alpha(l_2+[N,L],n_2+[N,L])]\]
and
\[\alpha(l_1+[N,L] \ , \ [n_1,n_2]+[N,L])\]
\[=\alpha(^{n_2+[N,L]} l_1+[N,L] \ , \ n_1+[N,L]) - \alpha(^{n_1+[N,L]}l_1+[N,L] \ , \ n_2+[N,L]).\]
Thus $\alpha$ is a Lie pairing  and  induces the  homomorphism
\[\bar{\alpha}: \frac{L}{[N,L]} \otimes  \frac{N}{[N,L]} \longrightarrow \frac{L\otimes N}{M}\] such that $\bar{\pi} \circ \bar{\alpha}=\bar{\alpha} \circ \bar{\pi}= 1$. Therefore \[\frac{L}{[N,L]} \otimes  \frac{N}{[N,L]} \simeq \frac{L \otimes N}{M}\] and the result follows.
\end{proof}

The assumption that a Lie algebra must be of finite dimension appears now for the first time. The main results of the present section will deal with this kind of Lie algebras only.

\begin{rem}\label{extra1}  Lemma \ref{1} and Corollary \ref{2} are true in particular when we replace the assumption that $N/[N,L]$ is free with $N/[N,L]$ of finite dimension. The above conditions are satisfied, for instance, when  $L$ is of finite dimension.
\end{rem}

Firstly, we describe $\Gamma(N/[N,L])$ (in \eqref{diag}) via a suitable isomorphism with $L \square N$.

\begin{cor}\label{3}
Let $N$ be an ideal of a Lie algebra $L$ such that $N/[N,L]$ is  of finite dimension. Then
\[\Gamma\left(\frac{N}{[N,L]}\right) \simeq L \ \square \ N  \simeq  \frac{L}{[N,L]} \ \square \ \frac{N}{[N,L]}.\]
\end{cor}
\begin{proof}
By Lemma \ref{1}, $\mathrm{Im} \ \psi = L \square N \simeq \Gamma(N/[N,L])$. On the other hand, \[\Gamma\left(\frac{N}{[N,L]}\right) \simeq  \frac{L}{[N,L]} \ \square \ \frac{N}{[N,L]}.\]
Thus
$L\square N \simeq L/[N,L] \ \square \ N/[N,L]$, as required.
\end{proof}

Secondly, we describe better one of the  isomorphisms in Corollary \ref{3}.

\begin{lem}\label{4}
Let $N$ be an ideal of a Lie algebra $L$ such that $N/[N,L]$ is  of finite dimension. Then
the natural epimorphism  \eqref{epi2} is  injective and so it is an isomorphism.
\end{lem}
\begin{proof}
By Corollary $\ref{3}$, we have   $\dim L \ \square \ N =\dim L/[N,L] \ \square \ N/[N,L]$ and so \eqref{epi2} is an isomorphism, because it is an epimorphism of abelian Lie algebras of same dimension.
\end{proof}

Now we begin to look for information on the bases of those Lie algebras, which will be involved in the main results of the present section. The abelian case plays a fundamental role and is discussed in the next result.

\begin{prop}\label{basis}Let $N= \langle x_1, x_2, \ldots , x_m \rangle$ be an ideal of dimension $m$ of an abelian Lie algebra $L = \langle  x_1, x_2, \ldots, x_m, y_{m+1}, y_{m+2}, \ldots, y_n \rangle$ of dimension $n$. Then
\[L \otimes N \simeq (L \ \square \ N) \  \oplus   \ \langle y_j\otimes x_t \ | \ 1\leq t\leq m, \ m+1 \leq j \leq n \rangle.\]
\end{prop}
\begin{proof}
From the  decomposition of finite dimensional abelian Lie algebras (see \cite{knapp}) in one dimensional ideals, we have \[L \simeq \bigoplus^{n}_{i=m+1} \langle y_i \rangle \ \oplus \  \ \bigoplus^{m}_{j=1}\langle x_j \rangle, \ \mathrm{where} \ N \simeq \bigoplus^{m}_{j=1}\langle x_j \rangle\] and so
\[L \otimes N \simeq \langle x_i \otimes x_t + x_t \otimes x_i, x_i\otimes x_i \ | \ 1 \leq i < t \leq m \rangle\]
\[\oplus \ \langle y_j \otimes x_t \ | \ 1 \leq t \ \leq m, m+1 \leq j \leq n\rangle.\]
Now the result follows from the fact that
\[L \ \square \ N \simeq \langle x_i\otimes x_t + x_t \otimes x_i, x_i \otimes x_i \ | \ 1\leq i <t \leq m\rangle .\]
\end{proof}

We are ready to prove the first main result of the present section.

\begin{thm}\label{6}
Let $N$ be an ideal of a  Lie algebra $L$ such that $L/[N,L]$ is of dimension $n$ with basis  $\{ \bar{x}_{_1},\bar{x}_{_2}, \ldots,\bar {x}_{_m},\bar{y}_{_{m+1}}, \bar{y}_{_{m+2}}, \ldots,\bar{y}_{_n}\}$  and $N/[N,L]$ of dimension $m$ with basis $\{ \bar{x}_{_1},\bar{x}_{_2}, \ldots,\bar {x}_{_m}\}$. Then
\[\frac{L}{[N,L]} \ \otimes \ \frac{N}{[N,L]} \simeq  \left(\frac{L}{[N,L]} \ \square \ \frac{N}{[N,L]}\right) \ \oplus \ \langle \bar{y_j} \otimes \bar{x_i} \ |\ 1\leq i \leq m, m+1\leq j \leq n \rangle.\]
\end{thm}
\begin{proof}
We do two observations. The first is that  $[L/[N,L],N/[N,L]]=0$. The second is that the actions of $L/[N,L]$ on $N/[N,L]$, and viceversa  of $N/[N,L]$ on $L/[N,L]$,  are  trivial, compatible and by conjugation. Then we may conclude that
\[\frac{L}{[N,L]} \ \otimes \ \frac{N}{[N,L]} \simeq {\left(\frac{L}{[N,L]}\right)}^{ab} \ \otimes_{\mathbb{Z}} \  \frac{N}{[N,L]}.\] From  Corollary \ref{2}, we have
\[{\left(\frac{L}{[N,L]}\right)}^{ab} \ \otimes_{\mathbb{Z}} \  \frac{N}{[N,L]} \simeq \left({\left(\frac{L}{[N,L]}\right)}^{ab} \ \square \  \frac{N}{[N,L]}\right) \ \oplus \  \left({\left(\frac{L}{[N,L]}\right)}^{ab} \ \wedge \  \frac{N}{[N,L]}\right),\]
where the abelian Lie algebra factor $(L/[N,L])^{ab} \ \wedge \  N/[N,L]$ admits a basis exactly of the form $\{ \bar{y_j} \otimes \bar{x_i} \ |\ 1\leq i \leq m, m+1\leq j \leq n \}$ while the other abelian Lie algebra factor is $(L/[N,L])^{ab} \ \square \  N/[N,L] \simeq L/[N,L] \ \square \  N/[N,L]. $ Then the result follows from Proposition \ref{basis}.
\end{proof}

Of course, Theorem \ref{6} is true when the entire $L$ is of finite dimension, not only its factor $L/[N,L]$. The second main result of this section may be formulated again with the restriction of finite dimension on the factor $L/[N,L]$.
It describes a pure splitting of the operator $\otimes$  with respect to the operators $\square$ and $\wedge$.

\begin{thm}\label{7}
Let $N$ be an ideal of a  Lie algebra $L$ such that $L/[N,L]$ is of finite dimension. Then
 \[ L\otimes N \simeq  (L \ \square \ N) \oplus (L \wedge N).\]
\end{thm}
\begin{proof}
Following the notations of Theorem \ref{6}, we note that
\[\frac{L}{[N,L]} \otimes  \frac{N}{[N,L]} \simeq \left(\frac{L}{[N,L]} \ \square \ \frac{N}{[N,L]}\right) \oplus \langle\bar{y_j}\otimes \bar{x_i} \ |  \ 1 \leq i \leq m, m+1\leq j \leq n\rangle.\]
On the other hand, $\pi_|$ in \eqref{epi2} is an isomorphism by Lemma \ref{4} and so
\[\left(\frac{L}{[N,L]} \ \square \ \frac{N}{[N,L]}\right) \oplus \langle \bar{y_j}\otimes \bar{x_i}  \ | \ 1\leq i \leq m, m+1\leq j \leq n \rangle\]
\[=\pi_| ((L \square N) \  + \ \langle y_j \otimes x_i \ | \ 1 \leq i \leq m, m+1\leq j \leq n \rangle ),\]
 where $\{\bar{x}_{_1},\bar{x}_{_2}, \ldots, \bar {x}_{_m} \}$ is a  basis of $N/[N,L]$  and $\{\bar{x}_{_1},\bar{x}_{_2}, \ldots,\bar {x}_{_m}, \bar{y}_{_{m+1}}, \bar{y}_{_{m+2}}, \ldots, \bar{y}_{_n}\}$ of $L/[N,L]$. Following the notation adopted in Lemma \ref{5}, we find that
\[L\otimes N = (L\square N) + \langle y_j \otimes x_i \ | \ 1\leq i \leq m, m+1\leq j \leq n\rangle + M.\] Now \eqref{epi2}  maps  all the elements of $(L \ \square \ N) \cap (\langle y_j \otimes x_i \ | \ 1\leq i \leq m, m+1\leq j \leq n\rangle + M)$ to  zero,  and so
\[L\otimes N \simeq  (L \ \square \ N) \ \oplus \ (\langle y_j \otimes x_i \ | \ 1\leq i \leq m, m+1\leq j \leq n\rangle + M).\]
Since we may easily check that $ L\wedge N \simeq  \langle y_j\otimes x_i \ | \ 1\leq i \leq m,m+1\leq j \leq n\rangle +M $, the result follows.
\end{proof}

The role of the Lie algebra $J_2(L,N)$ has been investigated in \cite{ellis2} when $N=L$ and it is related to the homotopy theory in the sense of \cite[Theorems 27, 28]{ellis2}. There is not a version of \cite[Theorems 27, 28]{ellis2}, when $N \neq L$, even if some ideas can be found in \cite[Theorems 1, 2, 4, 5]{ellis4} for the case of groups. This emphasizes the following consequence of  Theorem \ref{7}.

\begin{cor}\label{8}
Let $N$ be an ideal of a  Lie algebra $L$ such that $L/[N,L]$ is of finite dimension. Then
\[J_2(L,N) \simeq (L \ \square \ N)  \oplus M(L,N).\]
\end{cor}
\begin{proof}
By Theorem \ref{7},
\[J_2(L,N)=((L \ \square N) \oplus ( \langle y_j \otimes x_i \ | \ 1\leq i \leq m, m+1\leq j \leq n\rangle +M )) \cap J_2(L,N)\]
\[=(L \ \square \ N) \oplus (\langle y_j\otimes x_i \ | \ 1\leq i \leq m, m+1\leq j \leq n\rangle +M) \cap J_2(L,N).\]
The rest follows by Corollary \ref{2}; specifically by  $J_2(L,N)/ (L \ \square  \ N)  \simeq M(L,N)$.
\end{proof}

\section{K\"unneth-type formulas}

Dealing with homology of algebraic structures, there are some standard results, which help to understand the splitting of the second homology Lie algebra (with integral coefficients) with respect to  direct sums.  Of course, one may ask whether similar formulas happen for other operators, and not necessarily for the sum with respect to the second homology functor.
We list some of these splittings in our context of investigation. They are often called \textit{K\"unneth type formulas}, due to the original version of K\"unneth (see \cite{knapp})

\begin{prop}[See  Propositions 15, 21, Theorem 35 of \cite{ellis2}]\label{folklore}  Two Lie algebras $H$ and $K$ satisfy the conditions:
\begin{itemize}\item[(i)] $\Gamma((H \oplus K)^{ab}) \simeq \Gamma(H^{ab}) \oplus \Gamma(K^{ab}) \oplus (H^{ab} \ \otimes_\mathbb{Z} \ K^{ab})$;
\item[(ii)]$M(H \oplus K)  \simeq M(H) \oplus M(K) \oplus (H^{ab} \ \otimes_\mathbb{Z} \ K^{ab});$
\item[(iii)]$J_2(H \oplus K) \simeq J_2(H) \oplus J_2(K) \oplus (H^{ab} \ \otimes_\mathbb{Z} \ K^{ab})\oplus (K^{ab} \ \otimes_\mathbb{Z} \ H^{ab}).$
\end{itemize}
\end{prop}

Another version of Proposition \ref{folklore} can be formulated, when we consider the \textit{free product} $H * K$ of $H$ and $K$ in the usual sense (see \cite{knapp}). Roughly speaking,  the ``nonlinear'' factor $H^{ab} \ \otimes_\mathbb{Z} \ K^{ab}$ in (ii) above vanishes in (ii) below, and so  there is a perfect splitting. Similarly,  $K^{ab} \ \otimes_\mathbb{Z} \ H^{ab}$ in (iii) above vanishes in (iii) below.

\begin{prop}[See Proposition 19 and its proof in \cite{ellis2}]\label{folklorebis}  Two Lie algebras $H$ and $K$ satisfy the conditions:
\begin{itemize}\item[(i)] $\Gamma((H * K)^{ab}) \simeq \Gamma(H^{ab}) \oplus \Gamma(K^{ab}) \oplus (H^{ab} \ \otimes_\mathbb{Z} \ K^{ab})$;
\item[(ii)]$M(H * K) \simeq M(H) \oplus M(K);$
\item[(iii)]$J_2(H * K) \simeq J_2(H) \oplus J_2(K) \oplus (H^{ab} \ \otimes_\mathbb{Z} \ K^{ab}).$
\end{itemize}
\end{prop}

What happens to Propositions \ref{folklore} and \ref{folklorebis} for pairs of Lie algebras ?
Of course, we expect to have generalizations. The answer is positive and can be found  in the context of groups in \cite{ellis3, ellis4, mas1},  even if some authors are beginning to investigate pairs of Lie algebras in this perspective (see \cite{ri1, ri2, sam, n2, n4}). This motivated us to prove the following result.

\begin{thm}\label{freeproducts}
Let $N_1$ and $N_2$ two ideals of two finite dimensional Lie algebras $L_1$ and $L_2$ such that $(L_1 * L_2, N_1 * N_2 )$ is a pair. Then
\begin{itemize}
\item[(i)]$J_2(L_1 * L_2, N_1 * N_2)  \simeq J_2(L_1, N_1) \oplus J_2(L_2, N_2)  \oplus \left(\frac{N_1}{[N_1,L_1]} \otimes \frac{N_2}{[N_2,L_2]}\right);$
\item[(ii)] $(L_1 * L_2) \square (N_1 * N_2)  \simeq (L_1 \oplus N_1) \square (L_2 \oplus N_2);$
\item[(iii)]$M(L_1 * L_2, N_1 * N_2)  \simeq M(L_1, N_1) \oplus M(L_2, N_2)$, whenever $N_1$ admits a complement in $L_1$ such that $[N_1,L_1]=[L_1,L_1]$.
\end{itemize}
\end{thm}

\begin{proof}
(i). The natural monomorphisms $l_1 \in L_1 \mapsto l_1 \in L_1 *L_2$,
$l_2 \in L_2 \mapsto l_2 \in L_1 *L_2$, $n_1 \in N_1 \mapsto n_1 \in N_1 * N_2$ and  $n_2 \in N_2 \mapsto n_2 \in N_1 * N_2$  induce the natural monomorphisms
$\iota : l_1 \otimes n_1 \in L_1 \otimes N_1 \mapsto  l_1  \otimes n_1  \in (L_1 * L_2) \otimes (N_1 * N_2)$
and
$j :  l_2 \otimes n_2 \in L_2 \otimes N_2 \mapsto  l_2  \otimes n_2  \in (L_1 * L_2) \otimes (N_1 * N_2)$ and so the following map
\[\mu : (l_1 \otimes n_1, l_2 \otimes n_2) \in (L_1 \otimes N_1) \oplus (L_2 \otimes N_2) \mapsto
\iota (l_1 \otimes n_1) + j (l_2 \otimes n_2) \in (L_1 * L_2) \otimes (N_1 * N_2)\]
restricts to an injective homomorphism
\[\zeta : J_2(L_1, N_1) \oplus J_2(L_2,N_2) \rightarrow J_2(L_1 * L_2, N_1 * N_2).\]
On the other hand,
\[\eta : \overline{n_1}   \otimes \overline{n_2}  \in \frac{N_1}{[N_1,L_1]} \otimes \frac{N_2}{[N_2,L_2]} \mapsto (n_1 \otimes n_2) + (n_2 \otimes n_1) \in J_2(L_1 * L_2, N_1 *N_2), \]
where $\overline{n_1}=n_1 + [N_1,L_1]$ and $\overline{n_2}=n_2 + [N_2,L_2]$, is a well defined epimorphism. Hence we consider $\xi =  \zeta \oplus \eta$, where
\[\xi : (x,y,z) \in J_2(L_1, N_1) \oplus J_2(L_2, N_2)  \oplus \left(\frac{N_1}{[L_1,N_1]} \otimes \frac{N_2}{[L_2,N_2]}\right) \longmapsto\]
\[\xi(x,y,z)=(\zeta \oplus \eta) (x,y,z)=\zeta(x,y) + \eta(z) \in  J_2(L_1 * L_2, N_1 *N_2). \]
We claim that $\xi$ is bijective. Firstly, we note that
\[p : (L_1 * L_2) \otimes (N_1 * N_2) \rightarrow (L_1 \otimes N_1) \oplus (L_2 \otimes N_2) \oplus (L_1 \otimes N_2) \oplus (L_2 \otimes N_2)\]
is a canonical epimorphism and so
\[\xi (x,y,z) = 0 \Leftrightarrow (p \circ \xi) (x,y,0)=  (p \circ \xi) (0,0, -z) \Leftrightarrow \xi (x,y,0)=0 \Leftrightarrow x=y=z=0. \]
This is enough to conclude that $\xi$ is injective. Now we show that $\xi$ is surjective.  It is clear
that $(L_1 * L_2) \otimes (N_1 * N_2)$ is finitely
presented, because we have generators and relations by definition of $\otimes$. Using $i$, $j$ and $\mu$,  $L_1 \otimes
N_1$, $L_2 \otimes N_2$, $L_1 \otimes N_2$, $L_2 \otimes
N_1$ (and their sums) are isomorphic to suitable ideals
of $(L_1 * L_2) \otimes (N_1 * N_2)$. Then we may consider $t \in (L_1 * L_2) \otimes (N_1 * N_2)$, $U=L_1 \otimes N_1$, $W=L_2 \otimes N_2$, $V= (L_1 \otimes N_2) + (L_2 \otimes N_1)$ and $t=u+v+w$ for some $u \in U$, $v \in V$, $w \in W$. Here the commutator map $\kappa_{L_1*L_2, N_1*N_2}=\kappa$ is defined by \[\kappa : t \in (L_1 * L_2) \otimes (N_1 * N_2) \mapsto \kappa(t)=a+b+c \in [N_1,L_1] \oplus [N_2,L_2] \oplus [L_1, L_2]\]
and is an epimorphism. Note that
$\kappa(t)=0 \Leftrightarrow a=b=0 $ and so $\kappa(t)=\kappa(v)=c=0$. This implies $u \in J_2(L_1,N_1)$ and $w \in J_2(L_2,N_2)$. We also note that $\kappa(v)$ is freely reduced in $[L_1,L_2]$, whenever $v$ does not contain words of the form $k_1=(n_1 \otimes n_2) + (n_2 \otimes n_1)$ or $k_2=(n_2 \otimes n_1) + (n_1 \otimes n_2)$ in it. This means that there is no loss of generality in assuming $v=h+k_i+l$ with $h,l \in V$, $i=1,2$ and $k_1, k_2$ as before. Note that $\kappa(h+l)=\kappa(v)=0$. Now one can do induction on the number of expressions of the form $k_1, k_2$ and note that $\eta$ is an epimorphism. This allows us to conclude that $\xi$ is surjective and so we have
\[J_2(L_1 * L_2, N_1 * N_2)  \simeq J_2(L_1, N_1) \oplus J_2(L_2, N_2)  \oplus \left(\frac{N_1}{[N_1,L_1]} \otimes \frac{N_2}{[N_2,L_2]}\right). \]

(ii). Having in mind  $\iota$ and $j$ of the previous step (i), we  may  deduce from \eqref{diag} the commutativity of the following diagram
\[\begin{CD}
 \Gamma\left(\frac{L_1 * L_2}{[N_1*N_2, L_1*L_2]} \right) {\overset{\lambda}\longrightarrow} J_2(L_1*L_2,N_1*N_2) \longrightarrow M(L_1 * L_2, N_1*N_2) \longrightarrow 0 \\
\uparrow \hspace{3.5cm}  \uparrow \hspace{3cm}   \uparrow  \\
\Gamma\left(\frac{L_1 \oplus L_2}{[N_1 \oplus N_2, L_1 \oplus L_2]} \right) {\overset{\nu}\longrightarrow} J_2(L_1 \oplus L_2, N_1 \oplus N_2) \longrightarrow M(L_1 \oplus L_2, N_1 \oplus N_2) \longrightarrow 0  \end{CD}\]
which has exact sequences as rows.
Since
\[\mathrm{Im} \ \lambda= \langle x \otimes x \ | \ x \in (L_1*L_2) \cap (N_1 *N_2)\rangle= (L_1 * L_2) \square (N_1 * N_2) \] and
\[\mathrm{Im} \ \nu= \langle (a,b) \otimes (a,b) \ | \  (a,b)  \in  (L_1 \otimes L_2) \oplus (N_1 \otimes N_2) \rangle= (L_1 \oplus L_2) \ \square \ (N_1 \oplus N_2),\]
our scope is to prove that $\mathrm{Im} \ \lambda \simeq \mathrm{Im} \ \nu$.

We note that
the epimorphism $L_1*L_2 \longrightarrow L_1 \oplus L_2$ induces
an epimorphism \[\beta :(L_1*L_2) \otimes (L_1 * L_2)\longrightarrow (L_1 \oplus
L_2) \otimes (L_1 \oplus L_2)\] which may be restricted to an epimorphism   $\beta_|$ from
$\mathrm{Im} \ \lambda$ onto $\mathrm{Im} \ \nu$. So, it is
enough to  find a left inverse for $\beta_|$. We proceed to do this.

Keeping the notations of the above step (i), the canonical homomorphism
\[\varphi : (L_1 \oplus N_1) \otimes (L_2 \oplus N_1) \otimes (L_1 \oplus N_2) \otimes (L_2 \oplus N_2) \rightarrow\]
\[ (L_1 \otimes N_1) \oplus (L_2 \otimes N_1) \oplus (L_1 \otimes N_2) \oplus (L_2 \otimes N_2)\]
induces the following epimorphism
\[\sigma : (L_1 \oplus N_1) \otimes (L_2 \oplus N_1) \otimes (L_1 \oplus N_2) \otimes (L_2 \oplus N_2) \rightarrow
 \left(\frac{L_1}{[N_1,L_1]} \otimes \frac{N_1}{[N_1,L_1]}\right) \oplus\]\[ \left(\frac{L_2}{[N_2,L_2]} \otimes \frac{N_1}{[N_1, L_1]}\right) \oplus \left(\frac{L_1}{[N_1,L_1]} \otimes \frac{N_2}{[N_2,L_2]}\right) \oplus \left(\frac{L_2}{[N_2,L_2]} \otimes \frac{N_2}{[N_2,L_2]}\right).\]
We denote by $\mathrm{proj}_2$ the projection of $\sigma$ onto the second factor. Now we have all that is necessary to define the left inverse of $\beta_|$, this is the map
\[\theta=\beta_| \circ \mathrm{proj}_2 \circ \sigma_| \circ \zeta \circ \varphi_|\]
where $\varphi_|$ denotes the restriction of $\varphi$ to $\mathrm{Im} \ \nu$ and $\sigma_|$ that of $\sigma$ again to $\mathrm{Im} \ \nu$. Now one can check that $\theta \circ \beta_| =1$ and the result follows.

(iii). Applying  \cite[Lemma 7 (iii)]{ellis2} twice, that is, the distributive law
of the form $(L \oplus M) \otimes N \simeq  (L \otimes
N) \oplus (M \otimes N)$ for three given Lie algebras $L$, $M$ and $N$, we get the isomorphism
\[(L_1 \oplus L_2) \otimes (N_1 \oplus N_2) \simeq (L_1 \otimes N_1) \oplus (L_2 \otimes N_1) \oplus (L_1 \otimes N_2) \oplus (L_2 \otimes N_2)\]
and this induces  $(L_1 \oplus L_2) \ \square \ (N_1 \oplus N_2)  \simeq  (L_1 \ \square \ N_1) \oplus (L_2 \ \square \ N_2) \oplus V,$
where $V= \langle (n_2 \otimes n_1)+(n_1 \otimes n_2) \ | \ n_1 \in N_1, n_2 \in N_2\rangle \subseteq (L_2 \otimes N_1) + (L_1 \otimes N_2).$ Since $N_1$ admits a complement in $L_1$, the homomorphism $\eta$ in (i) above extends to an isomorphism onto $V$. Therefore the result follows from (ii) above and the commutativity of \eqref{diag}.
\end{proof}

It is possible to describe the Lie algebra quotient $J_2(L_1 \oplus L_2, N_1 \oplus N_2)/J_2(L_1 * L_2, N_1 * N_2)  $ very well and this is done in the following result.

\begin{cor} Let $N_1$ and $N_2$ two ideals of two finite dimensional Lie algebras $L_1$ and $L_2$ such that $(L_1 * L_2, N_1 * N_2 )$ is a pair. Then
\[\frac{J_2(L_1 \oplus L_2, N_1 \oplus N_2)}{J_2(L_1 * L_2, N_1 * N_2)} \simeq  \frac{N_1}{[N_1,L_1]} \otimes \frac{N_2}{[N_2,L_2]}.  \]
\end{cor}

\begin{proof}
This follows from  Theorem \ref{freeproducts} and Corollary \ref{8}.
\end{proof}

\end{document}